\documentclass[12pt,draft]{amsart}
\usepackage[all]{xy}
\usepackage{amsfonts}
\usepackage{amssymb}

\usepackage[english]{babel}

\newtheorem{teo}{Theorem}
\newtheorem{lem}[teo]{Lemma}

\newtheorem{rmk}[teo]{Remark}

\begin{document}

\title[Twisted conjugacy in lamplighter-type groups]
{Twisted conjugacy in some lamplighter-type groups}

\author{Evgenij Troitsky}
\thanks{The work is supported by the Russian Science Foundation under grant 21-11-00080.}
\address{Moscow Center for Fundamental and Applied Mathematics, MSU Department,\newline
	Dept. of Mech. and Math., Lomonosov Moscow State University, 119991 Moscow, Russia}
\email{troitsky@mech.math.msu.su}

\begin{abstract}
For a restricted wreath product $G\wr \mathbb{Z}^k$, where $G$ is a finite abelian group, we determine (almost in all cases) whether this product has the $R_\infty$ property (i.e., each its automorphism has infinite Reidemeister number).
\end{abstract}

\maketitle
The \emph{Reidemeister number} $R(\varphi)$ of an automorphism $\varphi:\Gamma\to\Gamma$ is defined as the number of its 
Reidemeister (or twisted conjugacy) classes (i.e., classes with the respect to the equivalence relation
$x\sim gx\varphi(g^{-1})$). We study the problem of determination of the class of groups with the property
 $R_\infty$ (i.e., each its automorphism has infinite Reidemeister number) among restricted wreath products $G\wr \mathbb{Z}^k$, where $G$ is a finite abelian group. 
We will prove the following two theorems, which give the answer almost in all cases. A part of cases in the first theorem was proved in \cite{FraimanTroitsky}.

\begin{teo}\label{teo:cases}
Suppose, the prime-power decomposition of $G$ is $\oplus_i ({\mathbb Z}_{(p_i)^{r_i}})^{d_i}$. Then under each of the following 
conditions the corresponding wreath product $G\wr {\mathbb Z}^k$ admits an automorphism $\varphi$ with $R(\varphi)<\infty$, i.e. does not have the property $R_\infty$:
\begin{description}
\item[Case 1) (see \cite{FraimanTroitsky})] for all $p_i=2$ and $p_i=3$, we have  $d_i\ge 2$ (and is arbitrary for primes $p_i>3$);
\item[Case 2) (see \cite{FraimanTroitsky})] there is no $p_i=2$ and $k$ is even;
\item[Case 3) (see \cite{FraimanTroitsky})] for all $p_i=2$, we have  $d_i\ge 2$  and $k=4s$ for some $s$;
\item[Case 4)]  for all $p_i=2$, we have $d_i \ge 3$ and $k$ is even; 
\item[Case 5)]  for all $p_i=2$, we have $d_i \ge 2$, $d_i \ne 3$ and $k$ is even, $k \ge 4$.
\end{description}
\end{teo}

\begin{teo}\label{teo:inverse}
In the above notation, suppose that there is a summand ${\mathbb Z}_{p^i}$ of multiplicity one, where $p=2$ and $k$ is arbitrary or $p=3$ and $k$ is odd. Then $G\wr {\mathbb Z}^k$ has the $R_\infty$ property.
\end{teo}

As a by-product we obtain a correction of an inaccuracy in \cite{gowon1} (see Remark \ref{rmk:gowon}).

We refer to \cite{gowon1,SteinTabackWong2015,TroLamp,Fraiman,FraimanTroitsky}
for the previous advances in the problem and to \cite{FelshtynTroitskyFaces2015JGT}
for a general look and a description of importance and applications of the problem under consideration.

Recall some facts to be used in the proofs.
We have by definition, $G\wr \mathbb{Z}^k = \Sigma \rtimes_\alpha \mathbb{Z}^k$,  where  
$\Sigma$ denotes $\oplus_{x\in \mathbb{Z}^k} G_x$, and $\alpha(x)(g_y) =g_{x+y}$. Here $g_x$ is $g \in G \cong G_x$.

Denote by $C(\varphi):=\{g\in \Gamma \colon \varphi(g)=g\}$ the subgroup of $\Gamma$, formed by $\varphi$-fixed elements.
For an inner automorphism as well as for its restriction on a normal subgroup, we use the notation $\tau_g(x)=gxg^{-1}$ .

\begin{lem}[Prop. 3.4 in \cite{FelLuchTro}]\label{lem:Jab_fin}
Suppose, $\Gamma$ is a finitely generated residually finite group, $\varphi:\Gamma\to\Gamma$ is an automorphism, and 
$R(\varphi)<\infty$. Then $|C({\varphi})|< \infty$.
\end{lem}

\begin{lem}[\cite{FelHill,go:nil1}, see also \cite{polyc,GoWon09Crelle}]\label{lem:extensions}
Suppose, $\varphi:\Gamma\to \Gamma$ is an automorphism of a discrete group, $H$ is a normal $\varphi$-invariant subgroup of $\Gamma$, so  $\varphi$ induces automorphisms $\varphi':H\to H$ and $\widetilde{\varphi}:\Gamma/H \to \Gamma/H$. Then
	\begin{itemize}
		\item the projection $\Gamma\to \Gamma/H$ maps Reidemeister classes of $\varphi$ onto Reidemeister classes of $\widetilde{\varphi}$, in particular $R(\widetilde{\varphi})\le R(\varphi)$;

		\item if $|C(\widetilde{\varphi})|=n$, then $R(\varphi')\le R(\varphi)\cdot n$;

		\item if $C(\widetilde{\varphi})=\{e\}$, then 		each Reidemeister class of $\varphi'$ is an intersection of the appropriate Reidemeister class of $\varphi$ and $H$; 

		\item if $C(\widetilde{\varphi})=\{e\}$, then 
		$R(\varphi)=\sum_{j=1}^R R(\tau_{g_j} \circ \varphi')$, where $g_1,\dots g_R$ are some elements of $\Gamma$ such that
		$p(g_1),\dots,p(g_R)$ are representatives of all Reidemeister classes of $\widetilde{\varphi}$, 
$p:\Gamma\to \Gamma/H$ is the natural projection and $R=R(\widetilde{\varphi})$.
	\end{itemize}	
\end{lem}
 
For a semidirect product $\Sigma \rtimes_\alpha {\mathbb Z}^k $, one has by \cite{Curran2008} that automorphisms $\varphi':\Sigma\to\Sigma$
and $\overline{\varphi}: \Sigma \rtimes_\alpha {\mathbb Z}^k  /\Sigma \cong {\mathbb Z}^k \to {\mathbb Z}^k \cong \Sigma \rtimes_\alpha {\mathbb Z}^k  /\Sigma$ define an automorphism
$\varphi$ of $\Sigma \rtimes_\alpha {\mathbb Z}^k$ (generally not unique) if and only if
\begin{equation}\label{eq:maineq}
	\varphi'(\alpha(m)(h))=\alpha(\overline{\varphi}(g))(\varphi'(h)),\qquad h\in\Sigma,\quad m\in {\mathbb Z}^k.
\end{equation}

Since $\Sigma$ is abelian, by \cite[p.~207]{Curran2008} the mapping $\varphi_1$ defined as $\varphi'$ on $\Sigma$ and by $\overline{\varphi}$ on ${\mathbb Z}^k\subset \Sigma \rtimes {\mathbb Z}^k $ is still an automorphism.  
That is why, as it is proved in \cite{FraimanTroitsky}, 
one can assume that
\begin{equation}\label{eq:restric_on_sub}
{\mathbb Z}^k \subset A\wr {\mathbb Z}^k \mbox{ is $\varphi$-invariant and } \varphi|_{{\mathbb Z}^k}=\overline{\varphi}.
\end{equation}

\begin{lem}[Lemma 2.4 in \cite{FraimanTroitsky}]\label{lem:R_nee}
An automorphism $\varphi: G\wr {\mathbb Z}^k \to G \wr {\mathbb Z}^k$ has $R(\varphi)<\infty$ if and only if 
$R(\overline{\varphi})< \infty$ and $R(\tau_m \circ \varphi')<\infty$ for any $m \in {\mathbb Z}^k$ (in fact, it is sufficient to verify this for representatives
of Reidemeister classes of $\overline{\varphi}$).
\end{lem}

\begin{lem}[Lemma 2.5 in \cite{FraimanTroitsky}]\label{lem:how_to_de}
Suppose, $\overline{\varphi}:{\mathbb Z}^k \to {\mathbb Z}^k$ and $F:G\to G$ are automorphisms.
Define $\varphi'$ by
\begin{equation}\label{eq:how_to_def}
\varphi'(a_0)=(Fa)_0,\qquad \varphi'(a_x)=(Fa)_{\overline{\varphi}(x)}.
\end{equation}
Then (\ref{eq:maineq}) holds and $\varphi'$ is an automorphism of $G\wr {\mathbb Z}^k$.

Evidently the subgroups $\oplus G_x$, where $x$ runs over an orbit of $\overline{\varphi}$, are $\varphi'$-invariant direct summands of $\Sigma$.
\end{lem}

Also, one can verify (see \cite{FelshtynHill1993CM}) that,
for $\overline{\varphi}:{\mathbb Z}^k\to {\mathbb Z}^k$ defined by a matrix $M$, one obtains
\begin{equation}\label{eq:FHZk}
R(\overline{\varphi})=\# \mathrm{Coker} (\mathrm{Id} -\overline{\varphi})=|\det(E-M)|,
\end{equation}
if $R(\overline{\varphi})<\infty$, and $|\det(E-M)|=0$ otherwise.
It is well known that, for an automorphism $\varphi: A\to A$ of finite abelian group $A$, we have
\begin{equation}\label{eq:finabel}
R(\varphi)=|C(\varphi)|.
\end{equation}

\begin{proof}[Proof of Theorem \ref{teo:cases}]
In each of the  cases we will consider  a specific automorphism $\overline{\varphi}:{\mathbb Z}^k \to {\mathbb Z}^k$ with $R(\overline{\varphi})<\infty$ 
and define $\varphi':\Sigma\to \Sigma$ with properties as in Lemmas \ref{lem:R_nee} and \ref{lem:how_to_de}.
More precisely, we will have $R(\varphi')=1$.

\textbf{Case 4):} when  $d_i >2$ for $p_i=2$ and $k=2 t$.
In this case the construction starts as in \cite{TroLamp}: we take $\overline{\varphi}:{\mathbb Z}^{2t} \to {\mathbb Z}^{2t}$ to be the direct sum of $t$ copies of
$$
{\mathbb Z}^2 \to {\mathbb Z}^2,\quad \begin{pmatrix}
u\\ v
\end{pmatrix} \mapsto M \begin{pmatrix}
u\\ v
\end{pmatrix}, \qquad M=\begin{pmatrix}
0 & 1\\
-1 & -1 
\end{pmatrix}. 
$$ 
Then $M$ generates a subgroup of $GL(2,{\mathbb Z})$, which is isomorphic to ${\mathbb Z}_3$ (see \cite[p. 179]{Newman1972book}).
All orbits of $M$ have length $3$ (except of the trivial one) and the corresponding Reidemeister number $=\det (E-M)=3$ (by (\ref{eq:FHZk})).
Similarly for $\overline{\varphi}$: the length of any orbit is $1$  or $3$ and $R(\overline{\varphi})=3^t$. 

Now define $\varphi'$ as a direct sum of  actions for ${\mathbb Z}_q$, $q=(p_i)^{r_i}$, $p_i \ge 3$, and for $({\mathbb Z}_{2^{r_i}})^2$, $({\mathbb Z}_{2^{r_i}})^3$.

For $p_i \ge 3$ choose $m=m_i$ such that   
\begin{equation}\label{eq:condit_on_m_3}
m^3 \mbox{ and } 1-m^3 \mbox{  are invertible in }{\mathbb Z}_q.
\end{equation}
This can be done for $p_i\ge 3$: one can take $m=3$ for $p_i=7$ and $m=2$ in the remaining cases (and impossible for $p_i=2$).  
Define $\varphi'(\delta_0)=m \delta_0$, where $\delta_0$ is $1\in ({\mathbb Z}_q)_0$. Then, to keep (\ref{eq:maineq}) we need to define 
$\varphi'(\delta_g)=m \delta_{\overline{\varphi}(g)}$, where $\delta_g$ is $1\in ({\mathbb Z}_q)_g$, as in Lemma \ref{lem:how_to_de}. So, the corresponding subgroup $\oplus_{g \in {\mathbb Z}^k}({\mathbb Z}_q)_g \subset \Sigma$ is $\varphi'$-invariant and decomposed into infinitely many invariant summands 
$({\mathbb Z}_q)_g \oplus ({\mathbb Z}_q)_{\overline{\varphi}(g)}\oplus ({\mathbb Z}_q)_{\overline{\varphi}^2(g)}$
isomorphic to $({\mathbb Z}_q)^3$ (over generic orbits of $\overline{\varphi}$) and one summand $({\mathbb Z}_q)_0$ (over the trivial orbit).
Then the corresponding restrictions of $\varphi'$ and   $1-\varphi'$ can be written as multiplication by 
$$
\begin{pmatrix}
0  &  0 & m \\
m & 0 & 0\\
0 & m & 0
\end{pmatrix}, \quad
\begin{pmatrix}
1  &  0 & -m \\
-m & 1 & 0\\
0 & -m & 1
\end{pmatrix}, \quad \mbox{and }
m, \quad 1-m,
$$
respectively. The three-dimensional mappings are isomorphisms by (\ref{eq:condit_on_m_3}).  Since an element $\ell$ is not invertible in ${\mathbb Z}_{(p_i)^{r_i}}$ if and only if $\ell = u\cdot p_i$, the invertibility of one-dimensional mappings follows from 
(\ref{eq:condit_on_m_3}) and the factorization $1-m^3=(1-m)(1+m+m^2)$. (This construction gives a more explicit presentation of a part of proof 
of \cite[Theorem 4.1]{TroLamp})

Now pass to $2$-subgroup. 
Define $F_j:({\mathbb Z}_q)^j \to ({\mathbb Z}_q)^j$ ($j=2,\dots,5$) as
\begin{equation}\label{eq:F2F3}
F_2 = \begin{pmatrix}
			0 & 1\\
			1 & 1
			\end{pmatrix} , F_3 = \begin{pmatrix}
			0&0 & 1\\
			0&1 & 1\\
            1& 1& 1
			\end{pmatrix},
F_4 = \begin{pmatrix}
			0&0 & 0& 1\\
			1& 0&0 & 1\\
            0& 1& 0 & 1\\
            0& 0& 1& 1
			\end{pmatrix},
		F_5 = \begin{pmatrix}
		0&  0&	0 & 0 & 1 \\
		0&  0&	0 & 1 & 1 \\
		0&  0&	1 & 1 & 1 \\
        0&  1&  1 & 1 & 1\\
        1 & 1 & 1 & 1 & 1
		\end{pmatrix}.
\end{equation}

For components $({\mathbb Z}_{2^{r_i}})^3$, $({\mathbb Z}_{2^{r_i}})^4$ and $({\mathbb Z}_{2^{r_i}})^5$,  
(in accordance with Lemma \ref{lem:how_to_de})
define  $\varphi'$ by
\begin{equation}\label{eq:def_on_comp}
a_0 \mapsto F_3 a_0,\qquad a_0 \mapsto F_4 a_0,\qquad a_0 \mapsto F_5 a_0,
\end{equation}
respectively.
Then, for any orbit of $\overline{\varphi}$ of length $\gamma$, we wish to verify that $(F_5)^\gamma$ as a homomorphism   $({\mathbb Z}_{2^i})^5\to ({\mathbb Z}_{2^i})^5$ has no non-trivial fixed elements
(to apply (\ref{eq:finabel})).
One can verify (better to write a small program) that the order of $F_5$ $\mod 2$ is $31$ and the less powers have no non-trivial fixed elements (here $\gamma=1$ or $3$ and we are interested in powers $1$ and $3$ only). An easier calculation shows the same for $F_3$ and $F_4$ (order $\mod 2$ is $7$ for both of them).
To complete the construction, it remains to observe that any number $\ge 3$ can be presented as a combination of $3$, $4$ and $5$ with positive integer coefficients. 

Concerning $\tau_g\circ \varphi'$ we should observe the following general property for automorphisms constructed in this way. One has $\overline{\varphi}^3(x)=x$  and can easily verify that $\overline{\varphi}^2z+\overline{\varphi}z+z=0$. Thus
$$
(\tau_z \circ \varphi') (g_x) = (m g)_{\overline{\varphi}(x)+z}, \quad
(\tau_z \circ \varphi') (g_{\overline{\varphi}(x)+z})=(mg)_{\overline{\varphi}^2(x)+\overline{\varphi}z+z},
$$
$$
(\tau_z \circ \varphi')g_{\overline{\varphi}^2(x)+\overline{\varphi}z+z}=(mg)_{\overline{\varphi}^3(x)+\overline{\varphi}^2z+\overline{\varphi}z+z}=(mg)_x.
$$
Hence $\tau_z \circ \varphi'$ is defined by the same matrices as $\varphi'$, but on new invariant summands
$({\mathbb Z}_q)_x \oplus ({\mathbb Z}_q)_{\overline{\varphi}(x)+z}\oplus ({\mathbb Z}_q)_{\overline{\varphi}^2(x)+\overline{\varphi}z+z}$. 
Similarly for the trivial orbit and for the $2$-subgroup. 
The properties of $m$ and $F_j$ imply the absence of non-trivial fixed elements. Hence, $R(\varphi')=1$ by
(\ref{eq:finabel}).
Case 4) is proved.

\textbf{Case 5):} when  $d_i >1$, $d_i \ne 3$ for $p_i=2$ and even $k \ge 4$ (i.e., we exclude the case $k=2$ from the consideration).
Using the cyclotomic polynomial we can define 
an element of order 5 in $GL(4,{\mathbb Z})$ 
$$
M_4=
\begin{pmatrix}
0 & 0 & 0 & -1 \\
1 & 0  &  0 & -1 \\
0 & 1 &   0 & -1\\
0& 0&  1 & -1
\end{pmatrix}
$$
and similarly to $M_4$  an element $M_6$ of order 7 in $GL(6,{\mathbb Z})$ . 
For any even $k\ge 4$, let $M\in GL(k,{\mathbb Z})$ be the direct sum of an appropriate number $s$ of $M_4$'s and maybe one $M_6$.
Let $\overline{\varphi}:{\mathbb Z}^k \to {\mathbb Z}^k$ be defined by $M$.  One can calculate
$$
\det(M_4-E)=5, \qquad \det(M_6-E)=7, \qquad \det(M-E)=5^s \mbox{ or } 5^s 7.
$$
Hence, by (\ref{eq:FHZk}),  $R(\overline{\varphi})=5^s \mbox{ or } 5^s 7<\infty$.
The length of any orbit is some divisor of $5$ or  $35$,  hence an \emph{odd} number. 

For $p$-power components ${\mathbb Z}_{p^i}$ with $p>2$, we define
$\varphi'$   by $a_0 \mapsto (p-1) a_0$.  Then, for an orbit $u, \overline{\varphi} u, \dots, \overline{\varphi}^\gamma u$, we need to verify (for finiteness of $R(\varphi')$) that $(p-1)^\gamma$ as a homomorphism    ${\mathbb Z}_{p^i} \to {\mathbb Z}_{p^i}$ has no non-trivial fixed elements, i.e. $ (p-1)^\gamma \not\equiv 1 \mod p$. This is fulfilled because, for an odd $\gamma$, $(p-1)^\gamma -1 \equiv -2  \not\equiv 0 \mod p$.

For $2$-power components ${\mathbb Z}_{2^i}\oplus {\mathbb Z}_{2^i}$, we define
$\varphi'$  by $a_0 \mapsto F_2 a_0$.  
For $2$-power components $({\mathbb Z}_{2^i})^5$, we define
$\varphi'$  by $a_0 \mapsto F_5 a_0$ .  Then, for any orbit of $\overline{\varphi}$ of length $\gamma$, we need to verify that $(F_5)^\gamma$ as a homomorphism   $({\mathbb Z}_{2^i})^5\to ({\mathbb Z}_{2^i})^5$ has no non-trivial fixed elements.
The end of the proof for Case 5) is completely analogous to Case 4). 
\end{proof}

\begin{proof}[Proof of Theorem \ref{teo:inverse}]
Denote by $A_p$ the $p$-subgroup of $G$.
Suppose that $R(\varphi)<\infty$ and arrive to a contradiction. For this purpose consider the characteristic subgroup $\Sigma' =\oplus A' \subset \Sigma$, where $A' \subset A$ is formed by elements of order $p^{i-1}$ from $A_p$ and by all elements of the other $A_q$, $q\ne p$.
Then $(A\wr {\mathbb Z}^k) /\Sigma' \cong B\wr {\mathbb Z}^k$, where $B=B'\oplus {\mathbb Z}_p$ and $B'=\oplus_j  ({\mathbb Z}_{p^j})^{d_j}$, $j\ge 2$, $d_j \ge 1$.
Conserve the notation $\varphi$, $\overline{\varphi}$ etc. for the induced automorphisms. 
By Lemma \ref{lem:extensions}, $R(\varphi_B)\le R(\varphi)<\infty$, where $\varphi_B:B\wr {\mathbb Z}^k \to 
B\wr {\mathbb Z}^k$  is the induced automorphism.
Consider the subgroup $D=B''\oplus {\mathbb Z}_p \subseteq B$ formed by all elements of order $p$ in $B$, where $B''$ is formed by all elements of order $p$ in $B'$. Then $D\subseteq B$ and $\Sigma_D=\oplus D \subseteq \Sigma'$ are characteristic subgroups 
and $D\wr {\mathbb Z}^k$ is a subgroup of $B\wr {\mathbb Z}^k$ (generally not normal). 
By (\ref{eq:restric_on_sub}) $D\wr {\mathbb Z}^k$ is $\varphi$-invariant and we 
wish to prove that $R(\varphi_D)<\infty$, where $\varphi_D: D\wr {\mathbb Z}^k \to D\wr {\mathbb Z}^k$ is the restriction of $\varphi_B$.

For this purpose consider $ \Sigma_{B/D}=\oplus B/D$ and automorphism $\varphi'_*: \Sigma_{B/D}\to \Sigma_{B/D}$ induced in an evident way.
We claim that $\varphi'_*$ has  $|C({\varphi'_*})| < \infty$. To prove this, consider  the following commutative diagram
\begin{equation}\label{eq:diagr_for_D_1}
\xymatrix{0 \ar[r]& \Sigma_D \ar[r] \ar[d]_{\varphi'_D}& \Sigma_{B} \rtimes {\mathbb Z}^k=B\wr {\mathbb Z}^k \ar[r] \ar[d]_{\varphi_B}  & \Sigma_{B/D} \rtimes {\mathbb Z}^k \ar[r] \ar[d]^{\varphi_*}& 0\\
0 \ar[r]& \Sigma_D \ar[r] & \Sigma_{B} \rtimes {\mathbb Z}^k=B \wr {\mathbb Z}^k \ar[r]  & \Sigma_{B/D} \rtimes {\mathbb Z}^k \ar[r]  & 0.}
\end{equation}
Then $\Sigma_{B/D} \rtimes {\mathbb Z}^k$ is a finitely generated residually finite group and $R(\varphi_*)<\infty$. 
Hence by Lemma \ref{lem:Jab_fin}, we have $|C({\varphi'_*})|< \infty$. 
Then by Lemma \ref{lem:extensions}, we obtain  $R(\varphi'_D)\le R(\varphi'_B) \cdot  |C({\varphi'_*})|< \infty$. 
The same is true for $\tau_g \circ \varphi$. Thus, by  Lemma \ref{lem:extensions}, 
since $C(\overline{\varphi})=\{0\}$,
from the commutative diagram
\begin{equation}\label{eq:diagr_for_D}
\xymatrix{0 \ar[r]& \Sigma_D \ar[r] \ar[d]_{\varphi'_D}& \Sigma_{D} \rtimes {\mathbb Z}^k=D\wr {\mathbb Z}^k \ar[r] \ar[d]_{\varphi_D}  &  {\mathbb Z}^k \ar[r] \ar[d]^{\overline{\varphi}}& 0\\
0 \ar[r]& \Sigma_D \ar[r] & \Sigma_{D} \rtimes {\mathbb Z}^k=D \wr {\mathbb Z}^k \ar[r]  &   {\mathbb Z}^k \ar[r]  & 0.}
\end{equation}
(and the analogous diagrams for $\tau_g \circ \varphi$) we obtain $R(\varphi_D)<\infty$ as desired.

Now observe that $B''$ is invariant in $D$ for automorphisms coming from $B$ (as restrictions), because $(b,0)$ is $p$-divisible in $B$ while $(b',k)$ is not $p$-divisible, where $b\in B''$, $b'\in B''$, $k \in {\mathbb Z}_p$, $k\ne 0$. The same is true for $\Sigma_{B''}=\oplus B''$. In particular, it is invariant for $\tau_g$, i.e., normal in $ \Sigma_{D} \rtimes {\mathbb Z}^k$.  Hence we can consider the diagram
\begin{equation}\label{eq:diagr_for_B''}
\xymatrix{0 \ar[r]& \Sigma_{B''} \ar[r] \ar[d]_{\varphi'_{B''}}& \Sigma_{D} \rtimes {\mathbb Z}^k=D \wr {\mathbb Z}^k \ar[r] \ar[d]_{\varphi_D}  & \Sigma_{D/B''} \rtimes {\mathbb Z}^k
= {\mathbb Z}_p\wr {\mathbb Z}^k \ar[r] \ar[d]^{\varphi_p}& 0\\
0 \ar[r]& \Sigma_{B''} \ar[r] & \Sigma_{D} \rtimes {\mathbb Z}^k=D \wr {\mathbb Z}^k \ar[r]  & \Sigma_{D/B''} \rtimes {\mathbb Z}^k= {\mathbb Z}_p\wr {\mathbb Z}^k  \ar[r]  & 0.}
\end{equation}
Then $R(\varphi_D)<\infty$ implies $R(\varphi_p)<\infty$ (Lemma \ref{lem:extensions}). But in the cases under consideration $R(\varphi_p)=\infty$ by 
\cite{TroLamp}. A contradiction.
\end{proof}

\begin{rmk}\label{rmk:gowon}\rm
Note that in the proof of Theorem 3.6 of \cite{gowon1} there is an inaccuracy in the argument after (3.3) because $(G\wr {\mathbb Z})/(L\wr {\mathbb Z})$ is not $G/L$
but $\oplus_i G/L$ which is not finite and even not finitely generated. The proof there can be corrected in a way as in our theorem above.
\end{rmk}

\end{document}